\documentclass{amsart}
\usepackage{amsmath,amscd,latexsym, amssymb}
\newtheorem{thm}{Theorem}[section]

\newtheorem{lem}[thm]{Lemma}
\newtheorem{cor}[thm]{Corollary}
\newtheorem{exa}[thm]{Example}

\begin{document}

\title {A characterization of normal subgroups via n-closed sets}
\author{Ayman Badawi}
\address{American Univ of Sharjah, Dept of Math, Box 26666, Sharjah, UAE}
\email{abadawi@aus.edu} 
\begin{abstract}
Let $(G, *)$ be a semigroup, $D  \subseteq G$, and $n \geq 2$ be an
integer. We say that {\it $(D, *)$ is an $n$-closed subset of  $G$}
if $a_1* \cdots *a_n \in D$ for every $a_1, ... , a_n \in D$. Hence
every closed set  is a 2-closed set. The concept of $n$-closed sets
arise  in so many natural examples. For example, let $D$ be the set
of all odd integers, then $(D, +)$ is a 3-closed subset of
$(\mathbb{Z}, +)$ that is not a 2-closed subset of $(\mathbb{Z},
+)$. If  $K = \{1,4, 7, 10, ... \}$ , then $(K, +)$ is a $4$-closed
subset of $(\mathbb{Z}, +)$ that is not an $n$-closed subset of
$(\mathbb{Z}, +)$ for $n = 2, 3$. In this paper, we show that if
$(H, *)$ is a subgroup of a group $(G, *)$ such that $[H: G] = n <
\infty$, then $H$ is a normal subgroup of $G$ if and only if every
left coset of $H$ is an $n+1$-closed subset of $G$.
\end{abstract}

\maketitle{}

\section{intruduction}
In this paper, we introduce the concept of $n$-closed sets for some
integer $n \geq 2$. Let $(G, *)$ be a semigroup, $D  \subseteq G$,
and $n \geq 2$ be an  integer. We say that {\it $(D, *)$ is an
$n$-closed subset of  $G$} if $a_1* \cdots *a_n \in D$ for every
$a_1, ... , a_n \in D$. Hence every closed set  is a 2-closed set.
The concept of $n$-closed sets arise  in so many natural examples.
For example, let $D$ be the set of all odd integers, then $(D, +)$
is a 3-closed subset of $(\mathbb{Z}, +)$ that is not a 2-closed
subset of $(\mathbb{Z}, +)$. If  $K = \{1,4, 7, 10, ... \}$ , then
$(K, +)$ is a $4$-closed subset of $(\mathbb{Z}, +)$ that is not an
$n$-closed subset of $(\mathbb{Z}, +)$ for $n = 2, 3$. In the second
section of this paper paper, many basic properties of $n$-closed
sets are studied. For example, we show that if a finite set $D$ of a
group $(G, *)$ is an $n$-closed subset of $G$, then $D$ is a left
coset of a subgroup of $G$. In the third section, we give a
characterization of normal subgroups via $n$-closed sets. For
example, we show that if $(H, *)$ is a subgroup of a group $(G, *)$
such that $[H: G] = n < \infty$, then $H$ is a normal subgroup of
$G$ if and only if every left coset of $H$ is an $n+1$-closed subset
of $G$. Though we feel that the proofs of many results in this short
paper are elementary, we feel that  the whole idea is original and
it has not been considered in the literature.

Let $(G, *)$ be a group. If $H$ is a subset of $G$ and $H \not = G$,
then we write $H \subset G$. If $H$ is a subgroup of $G$, then $[H :
G]$ denotes the number of all distinct left cosets of $H$. If $a \in
G$ and $n \geq 1$ is an integer, then $a^n = a* \cdots *a$ (n
times),  $(a^n)^{-1}$ is the inverse of $a^n$ in $G$, and $\mid a
\mid$ denotes the order of $a$ in $G$. Let $D$ be a subset of $G$,
and $d_1, ... , d_k \in D$. Then $d_1* \cdots *d_k*D = \{d_1* \cdots
*d_k*d \mid  d\in D\}$. As usual, $\mathbb{R}, \mathbb{Q}$, and
$\mathbb{Z}$  will denote real numbers, rational numbers, and
integers,  respectively.

\section{Basic properties of $n$-closed sets}

\begin{thm}
\label{t2.1} Let $D$ be a finite subset of a group $(G, *)$. Suppose
that $D$ is not a $2$-closed subset (subgroup) of $G$, and for some
integer $n \geq 3$, $(D, *)$ is an $n$-closed subset of $(G, *)$.
Then

{\bf (1)}. $(d_1* \cdots *d_{n-2}*D, *)$ is a subgroup of $G$ for
every $d_1, ... , d_{n-2} \in D$. In particular, $(d^{n-2}*D, *)$ is
a subgroup of $G$   for every $d \in D$.

{\bf (2)}. Let $b, d_1, ... , d_{n-2} \in D$. Then $D = b*d_1*
\cdots *d_{n-2}*D$ (i.e., $D$ is a left coset of the subgroup $(d_1*
\cdots *d_{n-2}*D, *)$ of $(G, *)$). In particular, $b*d^{n-2}*D =
D$ for every $b, d \in D$.

{\bf (3)}. $d^{n-2}*D = b^{n-2}*D = d_1* \cdots *d_{n-2}*D$ for
every $d, b, b_1, ... , b_{n-2} \in D$.

\end{thm}
\begin{proof}  Suppose that $D$ is an $n$-closed subset of $G$ for some integer $n \geq 3$.

{\bf (1)}. Let $d_1, ... , d_{n-2} \in D$. Since $D$ is a finite
subset of the group $(G, *)$, we only need to show that $(d_1*
\cdots *d_{n-2}*D, *)$ is a $2$-closed subset of $G$. Let $a, b \in
d_1* \cdots *d_{n-2}*D$. Hence $a = d_1* \cdots *d_{n-2}*h_1$ and $b
= d_1* \cdots *d_{n-2}*h_2$ for some $h_1, h_2 \in D$. Since $D$ is
an $n$-closed subset of $G$, $h_1*d_1* \cdots *d_{n-2}*h_2 = c \in
D$, and thus $a*b = d_1* \cdots *d_{n-2}*h_1*d_1* \cdots
*d_{n-2}*h_2 = d_1* \cdots *d_{n-2}*c  \in d_1* \cdots *d_{n-2}*D$.

{\bf (2)}.  Let $b, d_1, ... , d_{n-2} \in D$. Since $D$ is an
$n$-closed subset of $G$, $b*d_1* \cdots *d_{n-2}*a \in D$ for every
$a \in D$. Since $D$ is a finite subset of $G$ and $b*d_1* \cdots
*d_{n-2}*a_1 = b*d_1* \cdots *d_{n-2}*a_2 $ for some $a_1, a_2 \in
D$ if and only if $a_1 = a_2$, we conclude that $D = b*d_1* \cdots
*d_{n-2}*D$.

{\bf (3)}. $d, b, b_1, ... , b_{n-2} \in D$. Since $D = b*d^{n-2}*D
= b*b^{n-2}*D = b*d_1* \cdots *d_{n-2}*D$ by (2), we conclude that
$d^{n-2}*D = b^{n-2}*D = d_1* \cdots *d_{n-2}*D$  for every $d, b,
b_1, ... , b_{n-2} \in D$.
\end{proof}

\begin{cor} \label{c2.01}
Let $(G, *)$ be a finite group and $D$ be a subset of $G$. Suppose
that $D$ is not a $2$-closed subset (subgroup) of $G$, and for some
integer $n \geq 3$, $(D, *)$ is an $n$-closed subset of $(G, *)$.
Then $H = d^{n-2}*D$ is a subgroup of $G$  for every $d \in D$ and
$D$ is a left coset of $H$.
\end{cor}

The following example shows that the hypothesis that $D$ is finite in Theorem \ref{t2.1} is crucial.

\begin{exa}
Let $G = (\mathbb{Z}, +)$, and $D = \{1, 3, 5, ..., \}$ be the set
of all positive odd numbers of $\mathbb{Z}$. Then $D$ is a 3-closed
subset of $\mathbb{Z}$, but $(a^{3-2} + D = a + D, +) $ is not a
subgroup of $G$  for every $a \in D$.
\end{exa}

In view of the proof of Theorem \ref{t2.1}, we have the following.

\begin{cor}
\label{c2.1} Let $(G, *)$ be a  a semigroup and $(D, *)$ be a an
$n$-closed subset of $(G, *)$ for some integer $n \geq3$. Then
$(d_1* \cdots *d_{n-2}*D), *)$ is a $2$-closed subset of $(G, *)$
for every $d_1, ... , d_{n-2} \in D$. In particular, $(d^{n-2}*D,
*)$ is a $2$-closed subset of $(G, *)$ for every $d \in D$ .
\end{cor}

\begin{thm}
\label{t2.2} Let $(G, *)$ be a group, $H \subset G$  be a subgroup
of $G$, and $L = a*H$ for some $a \in G \setminus H$. Suppose that
$(L, *)$ is an $n$-closed subset of $G$ for some integer $n \geq 2$,
and let $k \geq 2 $ be the least integer such that $(L, *)$ is a
$k$-closed subset of $G$. Then:

{\bf (1)}. $a^{n-1} \in H$, and hence  $n \geq 3$.

{\bf (2)}. $a*H = H*a = L$, and hence  $b*H = H*b = L$ for every $b \in L$.

{\bf (3)}. $a^{n-2}*L =  L*a^{n-2} = H$, and hence $b_1* \cdots
*b_{n-2}*L = L *b_1* \cdots *b_{n-2} =  H$ for every $b_1, ..., b_{n
- 2} \in L$.

{\bf (4)}  $a^{m} \in H$ for some integer $m > 0$ if and only if
$(k-1) \mid m$, and hence for every $d \in L$, $d^m \in H$ for some
positive integer $m$ if and only if $(k-1) \mid m$.

{\bf (5)}. $(L, *)$ is an $m$-closed subset of $G$  for some
positive integer $m$ if and only if $m = c(k-1) + 1$ for some
integer $c \geq 1$.
\end{thm}
\begin{proof}

{\bf (1)}. Since $L$ is an $n$-closed subset of $G$ and $a \in L$,
$a^n = a*h \in L$ for some $h \in H$, and thus $a^{n-1} = h  \in H$.
Since $a^{n-1} \in H$ and $a \in G \setminus H$, we have $n \geq 3$.

{\bf (2)}. Let $h \in H$. We show that $a*h = h_1*a$ for some $h_1
\in H$. Since $a^{n-1} \in H$ by (1),  $h_2 = (a^{n-1})^{-1} \in H$.
Since $L$ is $n$-closed, $(a*h)*(a*h*h_2)*a^{n-2} = a*h_3 \in L$ for
some $h_3 \in H$.  Thus $(a*h*h_2)*a^{n-2} = h^{-1}*h_3$, and hence
$(a*h*h_2)*a^{n-1} = h^{-1}*h_3*a$. Since $h_2 = (a^{n-1})^{-1}$, we
have $h^{-1}*h_3*a = (a*h*h_2)*a^{n-1} = a*h$. Since $h_1 =
h^{-1}*h_3 \in H$, $h_1*a =a*h$. Thus $a*H = H*a$. Let $b \in L$. We
show that $b*H = H*b = L$ for every $b \in L$. Since $b \in L$, $b =
a*h$ for some $h \in H$. Since $a*H = H*a$, $a*h = h_1*a$ for some
$h_1 \in H$. Thus $L = (a*h)*H = a*H = H*a = H*(h_1*a) = H*(a*h)$.

{\bf (3)}. Since $a^{n-1} \in H$ by (1),  we have $a^{n-2}*L =
a^{n-1}*H = H$. Since $a*H = H*a$ by (2), we have  $ H = a^{n-2}*L =
a^{n-1}*H = H*a^{n-1} = (H*a)*a^{n-2} =  L*a^{n-2}$. Let $b_1, ...,
b_{n-2} \in L$. Since $a*H = H*a$, we have $b_1* \cdots *b_{n-2} =
a^{n-2}*h$ for some $h \in H$. Thus $b_1* \cdots *b_{n-2}*L =
a^{n-1}*H = H*a^{n-1} = H*b_1 \cdots *b_{n-2} = H$.

{\bf (4)}. Suppose that $(k-1) \mid m$ for some positive integer
$m$. Since $a^{k-1} \in H$ by (1), we have $a^m \in H$.  Conversely,
suppose that $a^m \in H$ for some integer $m > 0$. Then $m = b(k -1
) + r$ for some integers \ $b, r \geq 0$ such that  $0 \leq r <
(k-1)$. We show that $r = 0$. Hence $a^m = a^{b(k-1) + r} =
a^{b(k-1)}*a^r \in H$. Since $a^{k-1} \in H$ and $a^{b(k-1)}*a^r \in
H$, we have $a^r \in H$. Let $d_1, ... , d_{r+1} \in L = a*H$. Since
$a*H = H*a$ by (2) and $a^r \in H$, there is an $h \in H$ such that
$d_1* \cdots *d_{r+1} = a^{r+1}*h = a*a^r*h \in a*H = L$. Thus $L$
is an $r+1$-closed subset of $G$, which is a contradiction since $r
+1 \leq (k-1)$ and $m \not = 0$. Hence $r = 0$ and $b \geq 1$.

{\bf (5)}.  Suppose that $L$ is an $m$-closed subset of $G$ for some
positive integer $m$.  Then $a^{m-1} \in H$ by (1). Hence $m -1 =
c(k-1)$ for some integer $c \geq 1$ by (4), and thus $m = c(k-1) +
1$. Conversely, suppose that $m = c(k-1) + 1$ for some integer $c
\geq 1$. Let $d_1, ... , d_m \in L$. Since $a*H = H*a$ by (2) and
$a^{m-1} \in H$ by (1), there is an $h \in H$ such that  $d_1*
\cdots *d_{m} = a^{m}*h = a*a^{m-1}*h \in a*H = L$. Thus $L$ is an
$m$-closed subset of $G$.

\end{proof}

In light of Theorem \ref{t2.2}[(1) and (2)] and the proof of  Theorem \ref{t2.2}(5), we have the following corollary.

\begin{cor}
\label{c2.2} Let $(G, *)$ be a group, $H \subset G$  be a subgroup
of $G$, and  $L = a*H$ for some $a \in G \setminus H$. Let $n \geq
3$. Then $(L, *)$  is an $n$-closed subset of $G$  if and only if
$a*H = H*a$ and $a^{n-1} \in H$.
\end{cor}

The proof of the following lemma is similar to the proof of the
well-known fact: Let $(G, *)$ be a group and $a \in G$ such that
$\mid  a \mid  = n < \infty$, then $\mid a^m \mid = n/gcd(m, n)$ for
every integer $m > 0$. Hence we omit the proof.

\begin{lem}
\label{l2.1} Let $(G, *)$ be a group, $H \subset G$  be a subgroup
of $G$, and  $a \in G \setminus H$. Suppose that $a^n \in H$  for
some  integer $n \geq 2$, and let $k \geq 2$ be the least integer
such that $a^k \in H$. Then for each  $m \geq 1$, we have $c =
k/gcd(m, k)$ is the least positive integer such that $(a^m)^c \in
H$. Furthermore, $(a^m)^f \in H$  for some integer $f \geq 1$ if and
only if $c \mid f$.
\end{lem}

\begin{thm}
\label{t2.3} Let $(G, *)$ be a group, $H \subset G$  be a subgroup
of $G$, and $L = a*H$ for some $a \in G \setminus H$. Suppose that
$(L, *)$ is an $n$-closed subset of $G$ for some integer $n \geq 3$,
and let $k \geq 3 $ be the least integer such that $(L, *)$ is a
$k$-closed subset of $G$. For each integer $m \geq 1$, let $c =
(k-1)/gcd(m, k-1)$. Then $a^m*H$ is a $c + 1$-closed subset of $G$.
Furthermore, $a^m*H$ is an $f$-closed subset of $G$ for some integer
$f \geq 1$ if and only if $f = bc + 1$ for some integer $b \geq 1$.
\end{thm}
\begin{proof}
Let $m \geq 1$, $K = a^m*H$,  and  $c = (k-1)/gcd(m, k-1)$. Since $L
= a*H$ is an $n$-closed subset of $G$ for some $n \geq 3$, we have
$a*H = H*a$ by Theorem \ref{t2.2} and thus $K = a^m*H = H*a^m$.
Since $k-1$ is the smallest integer such that $a^{k-1} \in H$, we
have $c = (k-1)/gcd(m, k-1)$ is the smallest integer such that
$(a^m)^c \in H$ by Lemma \ref{l2.1}. Hence $K = a^m*H$ is a
$c+1$-closed subset of $G$ by Corollary \ref{c2.2}. Thus $K = a^m*H$
is an $f$-closed subset of $G$ for some integer $f \geq 2$ if and
only $f = bc + 1$ by Theorem \ref{t2.2}(5)
\end{proof}

The following is a trivial example of $n$-closed sets.
\begin{exa}
Let $(G, *)$ be a group with at least two elements, and let $a$ be a
non-identity element of $G$. Suppose that $\mid a \mid = k < \infty$
for some integer $k \geq 2$. Then $\{a\}$ ia an $m$-closed subset of
$G$  for some $m \geq 3$ if and only if $m = bk + 1$ for some
integer $b \geq 1$.
\end{exa}

It is possible to have a group $(G, *)$ and a left coset $L$  of a
subgroup $H$ of $G$ such that for some integer $n \geq 2$,  $a^n \in
H$ and $a^{n+1} \in L$ for every $a \in L$, but yet $L$ is not an
$m$-closed subset of $G$ for every integer $m \geq 2$. We have the
following example.

\begin{exa}
Let $G = S_3$ be the permutation group on $3$ elements. Then  $H =
\{(1), (1 \ \ 2)\}$ is a subgroup of $G$, and $L = (1 \ \ 3) o H =
\{(1 \ \ 3), (1 \ \ 2 \ \ 3)\}$ is a left coset of $H$. Then $a^6
\in H$ and $a^7 \in L$ for every $a \in L$. Since $L = (1 \ \ 3) o H
= \{(1 \ \ 3), (1 \ \ 2 \ \ 3)\} \not = H o (1 \ \ 3)$, $L$ is not
an $m$-closed subset of $G$ for every integer $m \geq 2$ by
Corollary \ref{c2.2}.
\end{exa}

It is possible to have a group $(G, *)$ and a subgroup $H$ of $G$
such that for each $n \geq 3$, there is a left coset of $H$ that is
an $n$-closed subset of $G$, but it is not an $m$-closed subset for
each integer $m$, $2 \leq m < n$.  We have the following example.

\begin{exa}
Let $G = (\mathbb{Q}, +)$. Then $H = (\mathbb{Z}, +)$ is a subgroup
of $G$. Let $n \geq 3$. Then $L = \frac{1}{n-1} + \mathbb{Z}$ is a
left cost of $H$ that is an $n$-closed subset of $G$, but it is not
an $m$-closed subset for each integer $m$, $2 \leq m < n$.
\end{exa}

It is possible to have a subgroup $H$ of a group $G$ and a left
coset $L = a*H$ for some $a \in G\setminus H$ such that $a*H = H*a$,
but $L$ is not an $n$-closed subset of $G$ for every $n \geq 2$. We
have the following example.

\begin{exa}
Let $G = (\mathbb{R}, +)$. Then $H = (\mathbb{Z}, +)$ is a subgroup
of $G$,  $L = \sqrt{2} + \mathbb{Z}$ is a left coset of $H$, and
$\sqrt{2} + \mathbb{Z} = \mathbb{Z} + \sqrt{2}$, but $(L, +)$ is not
an $n$-closed subset of $G$ for each integer $n \geq 2$.
\end{exa}

\section{A characterization of normal subgroups}

In view of Corollary \ref{c2.2}, we have the following
characterization of normal subgroups via $n$-closed subsets.
\begin{thm}
\label{t3.1}

Let $(G, *)$ be a group and  $H \subset G$  be a subgroup of $G$.
The following statements are equivalent:

{\bf (1)}. $H$ is a normal subgroup of $G$ and for each $a \in G\setminus H$, there is an integer $n \geq 2$ such that $a^n \in H$.

{\bf (2)}.  For each $a \in G\setminus H$, there is an integer $m \geq 3$ such that $a*H$ is an $m$-closed subset of $G$.

\end{thm}

\begin{thm}
\label{t3.2} Let $(G, *)$ be a group and  $H \subset G$  be a
subgroup of $G$. Suppose that $[H : G] = n < \infty$. Then $H$ is a
normal subgroup of $G$ if and only if $a*H$ is an $n+1$-closed
subset of $G$ for each $a \in G\setminus H$.
\end{thm}
\begin{proof}
Suppose that $H$ is a normal subgroup of $G$. Since $G/H$ is a group
with exactly $n$ distinct elements, we have $a^n \in H$ for each $a
\in G\setminus H$. Thus we are done by Corollary \ref{c2.2}.
\end{proof}

\begin{cor}
Let $(G, *)$ be a finite group,  $H \subset G$  be a subgroup of
$G$, and $n = [H : G]$.  Then $H$ is a normal subgroup of $G$ if and
only if $a*H$ is an $n+1$-closed subset of $G$ for each $a \in
G\setminus H$.
\end{cor}

\end{document}